\documentclass[12pt,reqno]{amsart} %{{{
\usepackage{amsmath,amsthm,amsfonts,amssymb, mathtools,enumerate,xfrac}
\usepackage[
  colorlinks=true,
  citecolor=blue,
  ]{hyperref}
\usepackage{caption,subcaption}
\usepackage{tabularx}
\usepackage{multirow}
\usepackage[table]{xcolor}
\usepackage[text={6.5in,9.0in},centering]{geometry}

\usepackage{graphicx,url,psfrag}
\usepackage{tikz}
\usepackage{pgfplots}
\usetikzlibrary{
  decorations.pathreplacing,
  calc,
  decorations.fractals,
  through,
  shapes,
  patterns,
  arrows.meta,
  arrows,
  calligraphy,
  shapes.symbols,
  shapes.geometric,
  matrix,
  fit,
  shapes.geometric,
  intersections,
  positioning,
}
\usepgfplotslibrary{fillbetween}
\pgfplotsset{compat=1.16}

\theoremstyle{plain}
\newtheorem{theorem}{Theorem}[section]

\newtheorem{lemma}[theorem]{Lemma}

\theoremstyle{definition}

\newtheorem{remark}[theorem]{Remark}

\newtheorem{example}[theorem]{Example}

\numberwithin{equation}{section}

\usepackage{color}

\newcommand{\ud}{\ensuremath{\mathrm{d} }}
\newcommand{\R}{\mathbb{R}}

\newcommand{\InPrd}[1]{\left\langle #1 \right\rangle}

\begin{document}

\title[On the radius of self-repellent fractional Brownian motion]{On the radius of self-repellent \\fractional Brownian motion}

\author{Le Chen \and Sefika Kuzgun \and Carl Mueller \and Panqiu Xia}

\address{
  Department of Mathematics and Statistics\\
  Auburn University\\
  Auburn, AL, 36849, USA
}
\email{le.chen@auburn.edu}
\email{pqxia@auburn.edu}

\address{
  Department of Mathematics\\
  University of Rochester\\
  Rochester, NY, 513424, USA
 }
 \email{skuzgun@math.rochester.edu}
 \email{carl.e.mueller@rochester.edu}


\begin{abstract}
  We study the radius of gyration $R_T$ of a self-repellent fractional Brownian
  motion $\left\{B^H_t\right\}_{0\le t\le T}$ taking values in $\mathbb{R}^d$.
  Our sharpest result is for $d=1$, where we find that with high probability,
  \begin{equation*}
      R_T \asymp T^\nu, \quad \text{with $\nu=\frac{2}{3}\left(1+H\right)$.}
  \end{equation*}
  For $d>1$, we provide upper and lower bounds for the exponent $\nu$, but these
  bounds do not match.
\end{abstract}

\maketitle

\section{Introduction}

\textit{Self-avoiding random walks} are among the most extensively studied
models in statistical physics. A variant, the \textit{self-repellent walk} (also
known as the \textit{weakly self-avoiding walk}), provides a weaker version of
the self-avoiding walk. This variation adjusts the probability distribution of a
simple random walk by imposing penalties on paths with self-intersections. In
contrast, a self-avoiding walk is a random walk that strictly prohibits any
self-intersections. The \textit{Domb-Joyce model}~\cite{domb.joyce:72:cluster}
constitutes a discrete-time version of the self-repellent walk, while the
\textit{Edwards model}~\cite{edwards:65:statistical} provides a continuous
alternative, known as the \textit{self-repellent Brownian motion}. For an
in-depth treatment of self-avoiding walks, readers are directed to the
monograph~\cite{madras.slade:93:self-avoiding}, lecture
notes~\cite{bauerschmidt.duminil-copin.ea:12:lectures}, and the
survey~\cite{slade:19:self-avoiding} that highlights recent advancements in the
field.

This paper primarily focuses on the Edwards model associated with self-repellent
\textit{fractional Brownian motions} (fBm's). This particular model provides an
apt framework to analyze the properties of polymer molecules in good solvents,
as discussed in detail in~\cite{biswas.cherayil:95:dynamics}. Extensive
investigation of the self-repellent fBm has been undertaken
in~\cite{grothaus.oliveira.ea:11:self-avoiding,
bornales.oliveira.ea:13:self-repelling, bock.bornales.ea:15:scaling}, contingent
on the presence of a square-integrable (self-intersection) local time for the
fBm. However, it is known that the fBm with the \textit{Hurst parameter} $H$
might not possess a square-integrable local time if $dH \geq 1$---for instance,
if $d = 2$ and $H = 1/2$; see~\cite{rosen:87:intersection,
hu.nualart:05:renormalized}. As observed in~\cite{mueller.neuman:23:radius}, the
local time characteristic in the self-repellent Brownian motion is not
essential. Instead, one can work with the \textit{occupation measure}. In
particular, we consider the occupation measure of balls of radius 1.  Thus, we
penalize paths that come close to their past positions, rather than passing
through exactly the same points. Substituting local time with the occupation
measure should theoretically maintain the outcome, which is validated
specifically within dimension $1$. This insight facilitates the characterization
of self-repellent fBm's across all dimensions $d \geq 1$ and for the entire
range of the Hurst parameter $H \in (0,1)$.

Let $\left\{ B^H_t\right\}_{t\ge0}$ be a fractional Brownian motion with Hurst
index $H \in (0,1)$, taking values in $\mathbb{R}^d$.  That is,
$B^H_t=\left(B^{H,1}_t,\dots,B^{H,d}_t\right)$ where
$\left(B^{H,i}_\cdot\right)_{i=1}^{d}$ are independent one-dimensional
fractional Brownian motions with Hurst index $H$. Thus, each $B^{H,i}_\cdot$ is
a centered Gaussian process on $[0,\infty)$ with covariance
\begin{equation*}
  \mathbb{E}\left[B^{H,i}_s B^{H,i}_t\right]=\frac{1}{2}\left(t^{2H}+s^{2H}-|t-s|^{2H}\right).
\end{equation*}
We define the occupation time as follows:
\begin{equation}\label{E:LTy}
  L_T(y) \coloneqq \left|\left\{t\in[0,T]: B^H_t\in\mathbf{B}_1(y)\right\}\right|
         = \int_0^T \mathbf{1}_{\mathbf{B}_1(y)} \left(B^H_t\right) \ud t,
\end{equation}
where $|S|$ denotes the Lebesgue measure of the set $S$, and $\mathbf{B}_r(y)$
is the open ball in $\mathbb{R}^d$, centered at $y$, of radius $r > 0$. Define
\begin{equation}\label{E:ET}
  \mathcal{E}_T \coloneqq \exp\left(-\beta\int_{\mathbb{R}^d}L_T(z)^2\ud z\right)
\end{equation}
and for an event $A$, let
\begin{equation}\label{E:Qt}
  \mathbb{Q}_T(A) \coloneqq \frac{1}{Z_T}\mathbb{E}^{\mathbb{P}_T}\left[\mathbf{1}_A\mathcal{E}_T\right], \hspace{1cm}
  Z_T             \coloneqq \mathbb{E}^{\mathbb{P}_T}\left[\mathcal{E}_T\right].
\end{equation}
Then, under probability measure $\mathbb{Q}_T$, $\left\{B^H_t \colon 0\leq t
\leq T\right\}$ is a self-repellent fBm. \bigskip

In this paper, we will investigate the \textit{radius of gyration} $R_T$
(see~\cite{fixman:62:radius}) of the self-repellent fractional Brownian motion
$B_t^H$.
\begin{align}\label{E:GT}
  R_T \coloneqq \left[\frac{1}{T}\int_{0}^{T}\left|B_t^H-\overline{B}_T^H\right|^2\ud t\right]^{1/2} \quad \text{with} \quad
  \overline{B}_T^H \coloneqq \frac{1}{T}\int_{0}^{T}B_t^H\ud t.
\end{align}
It is worth noting that another customary radius is the \textit{mean square
end-to-end distance}, which is often seen in mathematical  papers:
\begin{equation*}
  \left(\mathbb{E}^{\mathbb{Q}_T}\left[\left|B^H_T\right|^2\right]\right)^{1/2},
  \qquad \text{where $\left|B_T^H\right| \coloneqq \sqrt{\left(B_T^{H,1}\right)^2 + \cdots + \left(B_T^{H,d}\right)^2}$\:.}
\end{equation*}

Physicists often base their reasoning on universality, namely the belief that
changing the details of a model will not affect its large-scale behavior. We
believe that the exact definition of the radius is unlikely to change our final
result. It is expected that, maybe up to a logarithmic correction,
\begin{align*}
  R_T \asymp T^{\nu}
\end{align*}
for some exponent $\nu$ depending on the dimension $d$ and the Hurst parameter
$H$. It has been conjectured in~\cite{bornales.oliveira.ea:13:self-repelling}
and~\cite{bock.bornales.ea:15:scaling} that
\begin{align}\label{E:conjecture}
   \nu = \frac{2(1+H)}{2+d}.
\end{align}

In the case of one dimension, the radius exhibits ballistic behavior, $\nu=1$,
for the self-repellent random walk, as shown in~\cite{bolthausen:90:on,
greven.hollander:93:variational}. For fBm, the corresponding result is as
follows: \medskip

\begin{theorem}[$d=1$]\label{T:d=1}
  Let $B^H$ be a one-dimensional fBm with Hurst parameter $H\in (0,1)$, and let
  $\mathbb{Q}_T$ and $R_T$ be defined as in~\eqref{E:Qt} and~\eqref{E:GT},
  respectively. Then, for any $\beta > 0$, there exist nonrandom constants
  $T_{\beta} \geq e$, and $C_*$, $C^*$, $C_{\ref{E:logzt}}>0$ such that the
  following inequality holds whenever $T \geq T_{\beta}$:
  \begin{align*}
    \mathbb{Q}_T \left( C_*  \beta^{1/3} T^{\frac{2(1+H)}{3}} \leq R_T \leq C^* \beta^{1/3} T^{\frac{2(1+H)}{3}} \right)
    \ge 1 - 2 \exp\left(-  C_{\ref{E:logzt}} \beta^{2/3} T^{\frac{2(2-H)}{3}}\right),
  \end{align*}
  where the constants $C_*$, $C^*$ are given in~\eqref{E:Cstar} and
  $C_{\ref{E:logzt}}$ in~\eqref{E:logzt}.
\end{theorem}
\medskip

In other words, the radius of the one-dimensional self-repellent fBm is
completely solved:
\begin{align*}
  R_T \asymp T^\nu \quad \text{with}\quad \nu=\frac{2}{3}\left(1+H\right)
\end{align*}
with high probability for large $T$. This theorem proves the conjectured
claim~\eqref{E:conjecture} for $d=1$, if we use the  radius of gyration. In
%particular, when $H = \sfrac{1}{2}$, $R_T \asymp T$, which coincides with the
particular, when $H = 1/2$, $R_T \asymp T$, which coincides with the
classical result for the Brownian motion/random walk in~\cite{bolthausen:90:on,
greven.hollander:93:variational}.

The analogous question in dimensions $d=2,3,4$ is completely open even in the
special case of self-repellent random walk/Brownian motion. In dimensions $d\geq
5$, the \textit{lace expansion} was successfully used to show that $\nu=1/2$ for
self-repellent random walk; see~\cite{hara.slade:92:self-avoiding,
brydges.spencer:85:self-avoiding}. The lace expansion is not expected to work in
the fBm case, since it requires the Markov property. We have the following
result for all dimensions: \medskip

\begin{theorem}[$d\ge 1$]\label{T:main}
  Let $B^H$ be a $d$-dimensional fBm with Hurst parameter $H\in (0,1)$, and let
  $\mathbb{Q}_T$ and $R_T$ be defined as in~\eqref{E:Qt} and~\eqref{E:GT},
  respectively. Then, for any $\beta > 0$, there exists some nonrandom constant
  $T_{\beta} \geq e$ such that the following inequality holds whenever $T \geq
  T_{\beta}$:
  \begin{gather}\label{E:main}
    \begin{aligned}
      \mathbb{Q}_T \left( C_*\: \underline{R}_T  \leq  R_T \leq C^*\: \overline{R}_T  \right)
      \geq  1 - 2 \exp\left(-  C_{\ref{E:logzt}} \gamma_{d,H} (\beta) F_{d,H} (T)\right).
    \end{aligned}
  \end{gather}
  In~\eqref{E:main}, $\gamma_{d,H}(\beta)$ and $F_{d,H}(T)$ are defined in
  Table~\ref{Tb:gamma-F} with
  \begin{align}\label{E:beta}
    \beta^{a,b} \coloneqq \beta^a \mathbf{1}_{\{0<\beta\le 1\}} + \beta^b \mathbf{1}_{\{\beta>1\}}.
  \end{align}
  The constants $C_*$ and $C^*$ are defined as follows:
  \begin{align}\label{E:Cstar}
    C_* \coloneqq \left(\frac{C_{\ref{E:bd_q<}}}{2 C_{\ref{E:logzt}}}\right)^{1/d} \quad \text{and} \quad
    C^* \coloneqq \left(\frac{2C_{\ref{E:logzt}}}{C_{\ref{E:bd_q>}}}\right)^{1/2},
  \end{align}
  with $C_{\ref{E:bd_q<}}$, $C_{\ref{E:bd_q>}}$, and $C_{\ref{E:logzt}}$ being
  the positive constants appearing in~\eqref{E:bd_q<}, \eqref{E:bd_q>},
  and~\eqref{E:logzt}, respectively. The bounds $\underline{R}_T$ and
  $\overline{R}_T$ in~\eqref{E:main} are equal to
  \begin{align}\label{E:R-bd}
    \begin{aligned}
      \underline{R}_T & = \underline{R}_T (d,H,\beta) \coloneqq \left(\frac{\beta  T^2}{\gamma_{d,H}(\beta) F_{d,H}(T)} \right)^{1/d}, \\
      \overline{R}_T  & = \overline{R}_T  (d,H,\beta) \coloneqq \left(\gamma_{d,H}(\beta) T^{2H} F_{d,H} (T) \right)^{1/2};
    \end{aligned}
  \end{align}
  see Table~\ref{Tb:R-bd} for their explicit expressions for various cases.
\end{theorem}

\begin{table}[htbp]
  \renewcommand{\arraystretch}{1.8}
    \newcommand{\shade}{\cellcolor{gray!40}}
  \begin{tabular}{|r|cc|cc|cc|}
    \hline
    \multirow{2}{*}{}                 & \multicolumn{2}{c|}{$dH < 1$} & \multicolumn{2}{c|}{$dH = 1$}                 & \multicolumn{2}{c|}{$dH > 1$} \\
                                      & \shade $d=1$                  & $d \ge 1$                                     & $H = 1/2$ and $d=2$            & $H < 1/2$    & $H \geq 1/2$           & $H < 1/2$ \\ \hline
    $\gamma_{d,H}(\beta) \coloneqq$ & \shade $\beta^{2/3}$                  & $\beta^{\frac{2 (1 - H)}{3 -  (d + 2) H}}$    & $\beta^{1/2,\, 2/3}$           & $\beta$      & $\beta^{2/3}$          & $\beta$   \\
      $F_{d,H}(T) \coloneqq$          & \shade $T^{\frac{2(2-H)}{3}}$ & $T^{1 + \frac{(1-2 H)(1 - dH)}{3 - (d+2) H}}$ & $T$                            & $T \log (T)$ & $T^{\frac{2(2-H)}{3}}$ & $T$       \\ \hline
  \end{tabular}
  \caption{Definitions of $\gamma_{d,H}(\beta)$ and $F_{d,H}(T)$, with the case
  $d=1$ being specially highlighted in the gray background. Recall that the
notation $\beta^{1/2,\, 2/3}$ is given in~\eqref{E:beta}.}
  \label{Tb:gamma-F}
\end{table}

In the one-dimensional case, $dH = H <1$, and thus
\begin{align*}
  \underline{R}_T = \overline{R}_T
  = \beta^{1/3} T^{\frac{2(1+H)}{3}};
\end{align*}
see the column $d=1$ in Table~\ref{Tb:dH} below for some concrete values.
Therefore, Theorem~\ref{T:d=1} is a direct corollary of Theorem~\ref{T:main}.
\bigskip

\begin{table}[h!tbp]
  \centering
  \renewcommand{\arraystretch}{1.5}

  \begin{subtable}{\textwidth}
    \centering
    \newcommand{\shade}{\cellcolor{gray!40}}
    \begin{tabular}{|c|cc|cc|}
      \hline
      \multirow{2}{*}{} & \multicolumn{2}{c|}{$dH < 1$}                              & \multicolumn{2}{c|}{$dH > 1$}                                                                \\
                        & \shade $d=1$                                               & $d \ge 1$                                                                                     & $H \geq 1/2$                                    & $H < 1/2$                    \\ \hline
      $\underline{R}_T$ & \shade \multirow{2}{*}{$\beta^{1/3} T^{\frac{2(1+H)}{3}}$} & $\left(\beta^{\frac{1 - dH}{3 - (d + 2)H}} T^{\frac{2 - 2dH^2}{3 - (d + 2)H}}\right)^{1/d} $  & $\left(\beta^{1/3}T^{2(1 + H)/3}\right)^{1/d} $ & $T^{1/d}$                    \\
      $\overline{R}_T$  & \shade                                                     & $\beta^{\frac{1 - H}{3 - (2+d)H}} T^{\frac{2 - (d-1)H - 2H^2}{3 - (d + 2)H}}$                 & $\beta^{1/3} T^{2(1+H)/3}$                      & $\beta^{1/2} T^{(1 + 2H)/2}$ \\ \hline
    \end{tabular}
    \caption{$d H \ne 1$}
    \label{Tb:R-bd-A}
  \end{subtable}

  \vspace{1em} % some vertical space between the two tables

  \begin{subtable}{\textwidth}
    \begin{tabular}{|c|cc|}
      \hline
      \multirow{2}{*}{} & \multicolumn{2}{c|}{$dH = 1$} \\
                        & $H = 1/2$ and $d=2$            & $H = 1/d$ and $d\ge 3$                      \\ \hline
      $\underline{R}_T$ & $\beta^{1/4, 1/6}\: T^{1/2}$   & $\left( T/\log(T)\right)^{1/d}$             \\
      $\overline{R}_T$  & $\beta^{1/4, 1/3}\: T$         & $\beta^{1/2} T^{(1 + 2H)/2} \sqrt{\log(T)}$ \\ \hline
    \end{tabular}
    \centering
    \caption{$d H =1$}
    \label{Tb:R-bd-B}
  \end{subtable}

  \caption{Explicit expressions of \( \underline{R}_T \) and \( \overline{R}_T
  \), as defined in~\eqref{E:R-bd}, for various values of $(d,H)$, with the case
$d=1$ being specially highlighted in the gray background. Recall that the
notation $\beta^{a,\, b}$ is given in~\eqref{E:beta}.}
  \label{Tb:R-bd}
\end{table}

The proof of Theorem~\ref{T:main}  builds on techniques from Mueller and
Neuman~\cite{mueller.neuman:23:radius}. Given $r>0$, define the following events
and probabilities:
\begin{alignat*}{4}%\label{E:Aq<}
  A^{(<)}_{r,T} & \coloneqq \left\{R_T\le r\right\}, \qquad & q^{(<)}_{r,T} & \coloneqq \mathbb{E}^{\mathbb{P}_T}\left[\mathbf{1}_{A_{r,T}^{(<)}}\mathcal{E}_T\right], \\
  A^{(>)}_{r,T} & \coloneqq \left\{R_T\ge r\right\}, \qquad & q^{(>)}_{r,T} & \coloneqq \mathbb{E}^{\mathbb{P}_T}\left[\mathbf{1}_{A_{r,T}^{(>)}}\mathcal{E}_T\right].
  %\label{E:Aq>}
\end{alignat*}

The following two lemmas, whose proofs are deferred to
Sections~\ref{sec:prf_lmma_1} and~\ref{sec:prf_lmma_2}, respectively, will be
used in the proof of Theorem~\ref{T:main}. \medskip

\begin{lemma}\label{lmm:q}
  For any $T\ge 0$, the following two inequalities hold:
    \begin{alignat}{4}\label{E:bd_q<}
    q^{(<)}_{r,T} & \leq \exp\left(-C_{\ref{E:bd_q<}} \frac{\beta\: T^2}{r^d}\right), & \quad & \text{for all $r\ge 1$,} \\
    q^{(>)}_{r,T} & \le \exp\left(-C_{\ref{E:bd_q>}} \frac{r^2}{\: T^{2H}}\right),                 & \quad & \text{for all $r>0$}.\label{E:bd_q>}
  \end{alignat}
  where
  \begin{align*}
  C_{\ref{E:bd_q<}} = C_{\ref{E:bd_q<}}(d) \coloneqq  \frac{9 \Gamma (1 + d/2)}{2^{5+d}\: \pi^{d/2}} \quad \text{and} \quad C_{\ref{E:bd_q>}} = C_{\ref{E:bd_q>}} (d, H) > 0.
  \end{align*}
\end{lemma}

\begin{lemma}\label{lmm:zt}
  There exist constants $C_{\ref{E:logzt}} = C_{\ref{E:logzt}}(d,H) > 0$ and
  $T_{\beta} > e$, such that for all $T \ge T_{\beta}$, $d\ge 1$, and $H\in
  (0,1)$, we have
  \begin{gather}\label{E:logzt}
    \log Z_T \geq - C_{\ref{E:logzt}}\: \gamma_{d,H}(\beta) F_{d,H} (T),
  \end{gather}
  where $\gamma_{d,H}(\beta)$ and $F_{d,H}(T)$ are defined in
  Table~\ref{Tb:gamma-F}.
\end{lemma}

\bigskip The paper is organized as follows. The proofs of Theorem~\ref{T:main}
and lemmas~\ref{lmm:q} and \ref{lmm:zt} are given in Section~\ref{Sec:prf}. We
then extend our discussions on our results in Section~\ref{Sec:illus}. Finally,
in the appendix, we include some known results for fBm's and the corresponding
Girsanov theorem.

\section{Proof of the main result}\label{Sec:prf}

We will defer the proofs of the lemmas to the next section, opting to initially
establish Theorem~\ref{T:main} through their application. Here, let us briefly
outline our strategy, which has been successfully employed in Brownian cases
in~\cite[Theorem~1.1]{mueller.neuman:23:radius}.

Let $a < b$ be real numbers, and let $c > 0$.  Recall that $\mathbb{Q}_T$ is
defined by~\eqref{E:Qt}. Let $X$ be a random variable. Then, to prove
\begin{align}\label{E_:stg-0}
  \mathbb{Q}_T \left(a \leq X \leq b\right) \geq 1 - 2 \exp (- c),
\end{align}
it suffices to show that
\begin{align*}
  \mathbb{Q}_T (X \leq a) & = \frac{\mathbb{E}^{\mathbb{P}_T} \left[\mathbf{1}_{\{X \leq a\}} \mathcal{E}_T\right]}{Z_T}  \leq \exp (-c), \shortintertext{and}
  \mathbb{Q}_T (X \geq b) & = \frac{\mathbb{E}^{\mathbb{P}_T} \left[\mathbf{1}_{\{X \geq b\}} \mathcal{E}_T\right]}{Z_T}  \leq \exp (-c).
\end{align*}
Additionally, the above two inequalities are ensured by the next three
inequalities:
\begin{gather}
  Z_T \geq \exp (- c),\label{E_:stg-1} \\[5pt]
  \mathbb{E}^{\mathbb{P}_T} \left[\mathbf{1}_{\{X \leq a\}} \mathcal{E}_T\right] \leq \exp (- 2c), \shortintertext{and}
  \mathbb{E}^{\mathbb{P}_T} \left[\mathbf{1}_{\{X \geq b\}} \mathcal{E}_T\right] \leq \exp (- 2 c). \label{E_:stg-3}
\end{gather}
Following this idea, we are ready to present the proof of Theorem~\ref{T:main}.

\begin{proof}[Proof of Theorem~\ref{T:main}]
  By the definition of $F_{d,H}$ in Table~\hyperref[Tb:gamma-F]{1}, it is clear
  that  the function $T^2/F_{d,H}(T)$ is monotone increasing on $[e,\infty)$
  with $\lim_{T\to\infty}T^2/F_{d,H}(T) = \infty$ for any $d \geq 1$ and $H \in
  (0,1)$. Hence, when $T$ is large enough, we can ensure that
  \begin{align*}
     r_{*} \coloneqq \left( \frac{C_{\ref{E:bd_q<}} \beta T^2}{ 2
     C_{\ref{E:logzt}} \gamma_{d,H}(\beta) F_{d,H} (T) } \right)^{1/d} \geq 1.
  \end{align*}
  Plugging the above $r_*$ to~\eqref{E:bd_q<} shows that
  \begin{align*}
    \mathbb{E}^{\mathbb{P}_T}\left[\mathbf{1}_{R_T \leq r_*}\right] = q^{(<)}_{r_*,T} \leq \exp\left(-C_{\ref{E:bd_q<}}\frac{ \beta T^2}{r_*^d}\right) = \exp\big(- 2C_{\ref{E:logzt}} \gamma_{d,H}(\beta) F_{d,H} (T)\big).
  \end{align*}
  Next, by choosing the following $r^*$ in~\eqref{E:bd_q>}
  \begin{align*}
    r^{*} \coloneqq
    \left(\frac{2C_{\ref{E:logzt}}\gamma_{d,H}(\beta) T^{2H} F_{d,H}(T)}{C_{\ref{E:bd_q>}}}\right)^{1/2},
  \end{align*}
  we have
  \begin{align*}
    \mathbb{E}^{\mathbb{P}_T}\left[\mathbf{1}_{R_T \geq r^*}\right]
    = q^{(>)}_{r^*,T} \leq \exp\left(- C_{\ref{E:bd_q>}}\, \frac{(r^*)^2}{T^{2H}}\right)
    = \exp\big(- 2C_{\ref{E:logzt}} \gamma_{d,H}(\beta) F_{d,H} (T)\big).
  \end{align*}
  Concerning~\eqref{E:logzt} in Lemma~\ref{lmm:zt}. We have justified all
  inequalities~\eqref{E_:stg-1}--\eqref{E_:stg-3}, and
  therefore~\eqref{E_:stg-0}, with $X = R_T$ defined as in~\eqref{E:GT}, $a =
  r_*$, $b = r_*$, and $c = C_{\ref{E:logzt}} \gamma_{d,H}(\beta) F_{d,H} (T)$.
  This proves~\eqref{E:main} with~$\underline{R}_T$ and $\overline{R}_T$ defined
  as in~\eqref{E:R-bd}. This completes the proof of Theorem~\ref{T:main}.
\end{proof}

\subsection{Upper bounds on \texorpdfstring{$q^{(<)}_{r,T},q^{(>)}_{r,T}$}{}---Proof of Lemma~\ref{lmm:q}}\label{sec:prf_lmma_1}

\textbf{(1)}~Fix an arbitrary $r\ge 1$. Suppose $R_T\ge r$. Then there must be a
time $t_1\in[0,T]$ such that $\left|B^H_{t_1}-\overline{B}^H_T\right|\ge r$.
This in turn implies that there exists $t_2\in[0,T]$ such that
$\left|B^H_{t_1}-B^H_{t_2}\right|\ge r$. Then the triangle inequality shows that
either $\left|B^H_{t_1}\right|\ge r/2$ or $\left|B^H_{t_2}\right|\ge r/2$. So we
conclude that $\sup_{t \in [0,T]} |B_t|\ge r/2$. Hence, concerning
$\mathcal{E}_T\le1$, and using~\cite[Theorem~4.1.1]{adler.taylor:07:random}, we
have
\begin{align*}
  q^{(>)}_{r,T}
  =   & \mathbb{E}^{\mathbb{P}}\left[\mathbf{1}_{A^{(>)}_{r,T}} \mathcal{E}_T\right]
  \le   \mathbb{P}\left(A^{(>)}_{r,T}\right)
  \le   \mathbb{P}\left(\sup_{t\in[0,T]}\left|B^H_t\right|\ge r/2\right)
  \le \exp\left(- C_{\ref{E:bd_q>}} \frac{r^2}{T^{2H}}\right),
\end{align*}
for some constant $C_{\ref{E:bd_q>}}>0$.
This proves~\eqref{E:bd_q>}.

\bigskip

\noindent \textbf{(2)~} Fix an arbitrary $r>0$. Recall that $L_T(x)$ and
$\mathcal{E}_T$ are defined in~\eqref{E:LTy} and~\eqref{E:ET}, respectively. We
claim that
\begin{align}\label{E:Claim}
  R_T \le r \quad \text{implies} \quad
  -\log\mathcal{E}_T = \int_{\mathbb{R}^d}L_T(x)^2 \ud x \ge 2 C_{\ref{E:bd_q<}} \frac{T^2}{(r+1)^d}.
\end{align}
As a consequence, for all $r \geq 1$,
\begin{align*}
  q^{(<)}_{r,T}
  =  \mathbb{E}^{\mathbb{P}}\left[\mathbf{1}_{A^{(<)}_{r,T}} \mathcal{E}_T\right]
  \le & \sup_{\omega\in A^{(<)}_{r,T}}\mathcal{E}_T(\omega)
  = \exp\left(-\beta \inf_{\omega\in A^{(<)}_{r,T}} \int_{\mathbb{R}^d}L_T(z,\omega)^2 \ud z\right)\\
  \leq & \exp \left(-  2 C_{\ref{E:bd_q<}} \frac{\beta T^2}{(r+1)^d} \right) \leq \exp \left(- C_{\ref{E:bd_q<}} \frac{\beta T^2}{r^d} \right).
\end{align*}
This confirms~\eqref{E:bd_q<}. Thus, it remains to prove~\eqref{E:Claim}. Assume
that $R_T \le r$. Notice that
\begin{equation*}
  \begin{split}
    r^2T & \ge \int_{0}^{T}\left|B^H_t-\overline{B}^H_T\right|^2\ud t
           \ge \int_{0}^{T}4r^2\mathbf{1}_{\{|B^H_t-\overline{B}^H_T| > 2r\}}\ud t \\
         & \ge 4r^2\left|\left\{t\in[0,T]: \left|B^H_t-\overline{B}^H_T\right| > 2r\right\}\right|,
  \end{split}
\end{equation*}
where $|\cdot|$ denotes Lebesgue measure. Therefore,
\begin{align*}
  & \left|\left\{t\in[0,T]:\left|B^H_t-\overline{B}^H_T\right|>2r\right\}\right| \le \frac{T}{4}  \shortintertext{and}
  & \left|\left\{t\in[0,T]:\left|B^H_t-\overline{B}^H_T\right|\le 2r\right\}\right| \ge \frac{3T}{4}.
\end{align*}
It follows that
\begin{equation}\label{eq:spend-time}
  \int_{\mathbf{B}_{(2r+1)}(\overline{B}^H_T)}L_T(y)\ud y \ge \frac{3T}{4}.
\end{equation}
Denote by $K_d \coloneqq \pi^{d/2}/\Gamma (1 + d/2)$ the volume of the unit ball in
$\mathbb{R}^d$. Then
\begin{equation*}
  C_{\ref{E:bd_q<}} = \frac{9}{2^{5 + d} K_d}, \quad \text{and} \quad K_d (2r+1)^d = \left|\mathbf{B}_{2r+1}\left(\overline{B}^H_T\right)\right|
                    = \int_{\mathbf{B}_{2r+1}\left(\overline{B}^H_T\right)}\ud y.
\end{equation*}
Now, by the Cauchy-Schwarz inequality and~\eqref{eq:spend-time},
\begin{equation*}%\label{eq:st-Cauchy-Schwarz}
  \begin{split}
    K_d (2r+1)^d\int_{\mathbb{R}^d}L_T(y)^2\ud y
    & \ge \int_{\mathbf{B}_{2r+1}(\overline{B}^H_T)}\ud y \int_{\mathbf{B}_{2r+1}(\overline{B}^H_T)}L_T(y)^2\ud y \\
    & \ge \left(\int_{\mathbf{B}_{2r+1}(\overline{B}^H_T)}L_T(y)\ud y\right)^2 \ge \frac{9T^2}{16}.
  \end{split}
\end{equation*}
It follows that
\begin{equation*}
  \int_{\mathbb{R}^d}L_T(y)^2\ud y
  \ge \frac{9T^2}{16 K_d (2r+1)^d}
  \ge \frac{9}{2^{4+d}} \times \frac{T^2}{K_d (r+1)^d} = 2 C_{\ref{E:bd_q<}} \frac{T^2}{(r + 1)^d}.
\end{equation*}
This proves claim~\eqref{E:Claim}. The proof of Lemma~\ref{lmm:q} is complete. \qed

\subsection{Lower bounds on \texorpdfstring{$Z_T$}{}---Proof of Lemma~\ref{lmm:zt}}\label{sec:prf_lmma_2}

Now we study the term $Z_T$ defined as in~\eqref{E:Qt}. Let $\mathbf{u}$ be a
unit vector in $\mathbb{\R}^d$, and let $\mathbb{P}_{T}^{\lambda},\mathcal{Q}_T$
be given as in \eqref{E:p-lambda}. That is, the new measure
$\mathbb{P}_{T}^{\lambda}$ adds a drift proportional to $\lambda$, and
$\mathcal{Q}_T$ is the Radon-Nikodym change of measure term, see Theorem
\ref{T:gir}.  The drift enforces the behavior we think the self-repellent
process should have. Then, we can write
\begin{equation*}
  Z_T = \mathbb{E}^{\mathbb{P}_T}[\mathcal{E}_T]
      = \mathbb{E}^{\mathbb{P}^\lambda_T}\left[\mathcal{E}_T
        \cdot\left(\mathcal{Q}_T(\lambda M)\right)^{-1}\right].
\end{equation*}
Applying Jensen's inequality to $\log Z_T$, we find
\begin{equation*}
  \log Z_T
  \ge \mathbb{E}^{\mathbb{P}_T^\lambda}
    \left[\log\left(\mathcal{E}_T\cdot
    \left(\mathcal{Q}_T(\lambda M)\right)^{-1}\right)\right]
  = - \left(I_1 + I_2\right),
\end{equation*}
where
\begin{equation*}
  I_1 \coloneqq \mathbb{E}^{\mathbb{P}_T^\lambda}\left[\beta \int_{\mathbb{R}^d}L_T(y)^2 \ud y\right] \quad \text{and} \quad
  I_2 \coloneqq \mathbb{E}^{\mathbb{P}_T^\lambda}\left[\log\left(\mathcal{Q}_T(\lambda M)\right)\right].
\end{equation*}
As a consequence, we need to establish upper bounds for both $I_1$ and $I_2$. In
the proof below, the generic constant $C>0$ may vary from line to line.

\bigskip\noindent\textbf{Upper bound for $I_1$.} Due to the Girsanov formula for
fBm's as stated in Theorem~\ref{T:gir}, under $\mathbb{P}^{\lambda}_T$,
$\{B_t^H\colon 0\leq t \leq T\}$ has the same distribution as $\{B_t^H + \lambda
t \mathbf{u} \colon 0\leq t \leq T\}$ under $\mathbb{P}_T$. Fixing $T > e$. Let
$g(t,\cdot)$ be the probability density of $B^H_t$ for all $0 \leq t \leq T$.
Since $B^H$ has stationary increments, if $0\le s<t$ then the probability
density of $B^H_t-B^H_s$ is $g(t-s,\cdot)$. Note that for
$x,y,z\in\mathbb{R}^d$, if $|x-z|<1$ and $|y-z|<1$, then $|x-y|<2$. So we have
\begin{align*}
  I_1 &\le C\beta \int_{0}^{T} \ud t \int_{0}^{T} \ud s\: \mathbb{E}^{\mathbb{P}_T}
        \left[\mathbf{1}_{\mathbf{B}_2(0)}(B^H_t-B^H_s+(t-s)\lambda \mathbf{u}) \right]  \\
      & \le C\beta  \int_{0}^{T}\ud t \int_{0}^{t} \ud s \int_{\mathbf{B}_2(0)} \ud z\: g(t-s,z-(t-s)\lambda \mathbf{u}) \\
      & \le C\beta T\int_{0}^{T}\ud r \int_{\mathbf{B}_2(0)}\ud z\:  g(r,z-r\lambda \mathbf{u}).
\end{align*}
Hence, choosing $\mathbf{u} = (1,0,\dots, 0)$, we have that
\begin{equation*}
  \begin{split}
    I_1 \le C \beta T\int_{0}^{T} \ud r
    \left[
      \int_{-2}^{2} \ud z_1 \: r^{-H} \exp\left(-\frac{(z_1-r\lambda)^2}{r^{2H}}\right)  \times
      \prod_{i = 2}^d \int_{-2}^{2}\ud z_i \: r^{-H}  \exp\left(-\frac{z_i^2}{r^{2H}}\right)
    \right].
  \end{split}
\end{equation*}
Notice that for all $r \geq 4/\lambda$, $z_1 \in [-2,2]$, we have $|z_1| \leq
r\lambda /2$ and thus
\begin{align*}
  \exp \left(-\frac{(z_1-r\lambda)^2}{r^{2H}}\right)
  \leq \exp\left(-\frac{1}{4} \lambda^2 r^{2(1-H)}\right).
\end{align*}
Therefore,
\begin{align*}
  \int_{-2}^{2} \ud z_1\: \exp  \left(-\frac{(z_1 - r\lambda)^2}{r^{2H}}\right)
  & \leq  \mathbf{1}_{\{0 \leq r < 4/\lambda \}} \int_{-2}^{2} \ud z_1\: \exp  \left(-\frac{(z_1-r\lambda)^2}{r^{2H}}\right) \\
  & \quad + 4 \times \mathbf{1}_{\{ r \geq 4/\lambda \}} \exp\left(-\frac{1}{4} \lambda^2 r^{2(1-H)}\right).
\end{align*}
If $\lambda \in (0,1]$, it follows that
\begin{align*}
\mathbf{1}_{\{0 \leq r < 4/\lambda \}} \int_{-2}^{2}\ud z_1\: \exp  \left(-\frac{(z_1-r\lambda)^2}{r^{2H}}\right)
 \leq & 4 \times \mathbf{1}_{\{0 \leq r < 4/\lambda \}} \\
 \leq & C_* \mathbf{1}_{\{0 \leq r < 4/\lambda \}} \exp\left(-\frac{1}{4} \lambda^2 r^{2(1-H)}\right),
\end{align*}
where
\begin{align*}
  C_* \coloneqq \sup_{\lambda \in (0,1]} \sup_{r \in (0,4/\lambda )} 4\exp\left(\frac{1}{4} \lambda^2 r^{2(1-H)}\right)
      = 4\exp \left(2^{2(1-2H)}\right).
\end{align*}
On the other hand, if $\lambda > 1$, then we can write
\begin{align*}
  \mathbf{1}_{\{0 \leq r < 4/\lambda \}} \int_{-2}^{2} \ud z_1 \: \exp \left(-\frac{(z_1-r\lambda)^2}{r^{2H}}\right)
   \leq & \mathbf{1}_{\{0 \leq r < 4/\lambda \}} \int_{-\infty}^{\infty}  \ud z_1 \: \exp \left(-\frac{(z_1-r\lambda)^2}{r^{2H}}\right) \\
    =   & C \mathbf{1}_{\{0 \leq r < 4/\lambda \}} r^{H}.
\end{align*}
Combining the above four cases shows that
\begin{align*}
  \int_{-2}^{2} \ud z_1 \: \exp\left(-\frac{(z_1-r\lambda)^2}{r^{2H}}\right)
  \leq C \left( \exp\left(- \frac{1}{4} \lambda^2 r^{2(1-H)}\right) + \mathbf{1}_{\{0 \leq r < 4/\lambda < 4\}} r^{H} \right).
\end{align*}
Notice that
\begin{align*}
  \int_{-2}^{2} \ud z \: r^{-H} \exp\left(-\frac{z^2}{r^{2H}}\right)
  \le \min\left(4 r^{-H}, \sqrt{\pi}\right) \le C \left(r^{-H} \wedge 1\right).
\end{align*}
Therefore, we can write
\begin{align*}
  I_1 \leq & C \beta T \int_0^T  \ud r \: r^{-H} \left(1 \wedge r^{- H}\right)^{d-1} \left( \exp\left(- \frac{1}{4}\lambda^2 r^{2(1-H)}\right) + \mathbf{1}_{\{0 \leq r < 4/\lambda < 4\}} r^{H} \right) \\
      \leq & C \beta T \left(I_{1,1} + I_{1,2} + I_{1,3} \right),
\end{align*}
where
\begin{align*}
  I_{1,1} & \coloneqq \int_0^1 \ud r\: r^{- H}\exp\left(-\frac{1}{4} \lambda^2 r^{2(1-H)}\right),      \\
  I_{1,2} & \coloneqq \int_1^{T} \ud r\: r^{- dH} \exp\left(-\frac{1}{4} \lambda^2 r^{2(1-H)}\right) , \\
  I_{1,3} & \coloneqq \mathbf{1}_{\{\lambda > 1\}}\int_0^{4/\lambda} \ud r  \leq C \left( 1\wedge \lambda^{-1}\right).
\end{align*}

Performing a change of variable $\frac{1}{4} \lambda^2 r^{2(1-H)} = s$, we can
write
\begin{align*}
  I_{1,1} = & C \lambda^{-1}\int_0^{\frac{\lambda^2}{4}}  s^{-\frac{1}{2}} e^{-s} \ud s \\
  \leq      & C \left( \mathbf{1}_{\{0 < \lambda \leq 1\}} \lambda^{-1}\int_0^{\frac{\lambda^2}{4}} \ud s \: s^{-\frac{1}{2}}  + \mathbf{ 1 }_{\{\lambda > 1\}} \lambda^{-1} \int_0^{\infty} \ud s \: s^{-\frac{1}{2}} e^{-s}  \right)
  \leq  C \left( 1\wedge \lambda^{-1}\right),
\end{align*}
and
\begin{align*}
  I_{1,2} = C \lambda^{-\frac{1- dH}{1 - H}}\int_{\frac{\lambda^2}{4}}^{\frac{1}{4}\lambda^2 T^{2(1 - H)} } \ud s \: s^{\frac{1 - dH}{2(1- H)}-1} e^{- s} .
\end{align*}
Therefore, we need only estimate $I_{1,2}$. Assume $T \geq 1$, otherwise
$I_{1,2} = 0$.

\bigskip\noindent\textit{Case I:~} If $dH < 1$, then $\frac{1- dH}{2(1 - H)} >
0$ and hence,
\begin{align*}
  I_{1,2}
  \leq C \lambda^{-\frac{1 - dH}{1 - H}} \int_0^{\infty}  \ud s  \: s^{\frac{1-dH}{2(1- H)} -1} e^{- s}
    =  C \lambda^{-\frac{1 - dH}{1 - H}}.
\end{align*}

\bigskip\noindent\textit{Case II:~} If $dH = 1$, then we have that $I_{1,2} \leq
C \left( \log T \wedge \lambda^{-2}\right)$, which is due to
\begin{align*}
  I_{1,2} & =   C \int_{\frac{\lambda^2}{4}}^{\frac{1}{4}\lambda^2 T^{2(1 - H)}} \ud s \: s^{- 1} e^{- s}
          \le   C \lambda^{-2} \int_{\frac{\lambda^2}{4}}^{\frac{1 }{4}  \lambda^2 T^{2 - 2H} } \ud s  \:  e^{- s}
          \le   C \lambda^{-2} \quad \text{and} \\
  I_{1,2} & \le C \int_{\frac{\lambda^2}{4}}^{\frac{1}{4}\lambda^2 T^{2(1 - H)}} \ud s \: s^{- 1}
          \le   C \left( \log \left( \frac{\lambda^2}{4} T^{2 - 2H} \right) - \log \left( \frac{\lambda^2}{4} \right) \right)
           =    C \log T.
\end{align*}

\bigskip\noindent\textit{Case III:~} If $d H > 1$, then we have
\begin{align*}
  I_{1,2}
  \leq C \min\left(\int_1^{\infty}\ud r \: r^{- dH} ,\: \lambda^{-2} \int_0^{\infty} \ud s  \: e^{-s} \right)
  =    C \left( 1\wedge \lambda^{-2}\right).
\end{align*}

As a consequence, with $T > e$, we can write
\begin{align}\label{E:Rate-I1}
  I_1 \leq I_1^*(\lambda) \coloneqq \begin{dcases}
    C \beta T \lambda^{-\frac{1 - dH}{1 - H}},                                   & dH < 1, \\
    C \beta T \left(\log(T)  \wedge \lambda^{-2} + 1 \wedge \lambda^{-1}\right), & dH = 1, \\
    C \beta T \left(1 \wedge \lambda^{-1}\right),                                & dH > 1.
  \end{dcases}
\end{align}

\bigskip\noindent\textbf{Upper bound for $I_2$.} As for $I_2$, we need the
Girsanov formula for martingales. Recall Theorem~\ref{T:mar} that for any unit
vector $\mathbf{u} \in \R^d$, $M = M^{\mathbf{u}}$ defined as in~\eqref{E:M} is
a square-integrable martingale. The classical Girsanov formula for martingales
implies that $\widetilde{M} \coloneqq \left\{\widetilde{M}_t = M_t + \langle M
\rangle_t \right\}$ is a martingale under the probability measure
$\mathbb{P}_T^{\lambda}$ given by~\eqref{E:p-lambda}. As a consequence,
\begin{align}\label{E:Rate-I2}
  I_2
  =   \mathbb{E}^{\mathbb{P}^\lambda_T} \left[\log\left(\mathcal{Q}_T(\lambda M)\right)\right]
  = & \mathbb{E}^{\mathbb{P}^\lambda_T} \left[\lambda M_t - \frac{1}{2} \lambda^2 \langle M \rangle_t\right]
  =   \mathbb{E}^{\mathbb{P}^\lambda_T} \left[\lambda \widetilde{M}_t + \frac{1}{2} \lambda^2 \langle M \rangle_t\right] \nonumber \\
  = & \frac{1}{2} \lambda^2 \mathbb{E}^{\mathbb{P}^\lambda_T} \left[\langle M \rangle_t\right] = \frac{1}{2} C_H \lambda^2 t^{2(1-H)} \eqqcolon I_2^*(\lambda).
\end{align}

\bigskip\noindent\textbf{Matching bounds for $I_1$ and $I_2$.~} Recall that
$\log (Z_t) \geq - (I_1 + I_2)$, and $I_1$ and $I_2$ are bounded by
$I_1^*(\lambda)$ in~\eqref{E:Rate-I1} and $I_2^*(\lambda)$ in~\eqref{E:Rate-I2},
respectively. Notably, $I_1^*$ is a decreasing function of $\lambda \in
(0,\infty)$, whereas $I_2^*$ is an increasing function in the same range.
Therefore, in order to optimize the lower bound for $\log (Z_T)$ based
on~\eqref{E:Rate-I1} and~\eqref{E:Rate-I2}, we need to find a suitable $\lambda$
such that $I_1^*(\lambda)$ and $I_2^*(\lambda)$ coincide up to a constant. In
the following, we omit the tedious computations required to identify the
appropriate $\lambda$, and choose\footnote{Here, we leverage the fact that for
  any $\beta > 0$, the expression $\beta^{\gamma_1} T^{\gamma_2}
  \log(T)^{\gamma_3}$ behaves like $T^{\gamma_2} \log(T)^{\gamma_3}$ for large
$T$, where $\gamma_1, \gamma_2, \gamma_3$ are arbitrary numbers, with the
constraint that $\gamma_2$ and $\gamma_3$ cannot both be zero.}
\begin{align*}
  \lambda = \begin{dcases}
  \beta^{\frac{1 - H}{3 - (d+2) H }} T^{- \frac{(1 - 2H)(1 - H)}{3 - (d + 2)H}},            & dH < 1,          \\
  \beta^{1/3} \mathbf{1}_{\{\beta\ge 1\}} +  \beta^{1/4} \mathbf{1}_{\{0 < \beta \leq 1\}}, & dH = 1, H = 1/2, \\
  \beta^{1/2} T^{H - 1/2} \sqrt{\log T},                                                    & dH = 1, H < 1/2, \\
  \beta^{1/2} T^{H - 1/2},                                                                  & dH > 1, H < 1/2, \\
  \beta^{1/3} T^{\frac{2H-1}{3}},                                                           & dH > 1, H \ge 1/2.
  \end{dcases}
\end{align*}
Then,~\eqref{E:logzt} follows immediately. The proof of Lemma~\ref{lmm:zt} is
complete. \qed

\section{Discussion}\label{Sec:illus}

In this section, we present some concrete examples and make some remarks on our
results. When $d=1$, our results are sharp. When $d \ge 2$, Theorem~\ref{T:main}
provides some nontrivial lower and upper bounds, as illustrated in
Table~\ref{Tb:exponents} and Figure~\ref{F:nu-H}. One may compare our results
with the Brownian motion case given in~\cite{hara.slade:91:critical}, as
detailed in the next example: \medskip

\begin{example}[Brownian motion case] In the Brownian motion case, as seen in
  Table~\ref{Tb:exponents} (a), if $H = 1/2$, we have
  \begin{align*}
    T^{1/d}\lesssim R_T \lesssim T, \quad \text{for large $T$.} %\tag{$H=1/2$}
  \end{align*}
  This partially confirms the guess
  in~\cite[Equation~(1.30)]{bauerschmidt.duminil-copin.ea:12:lectures} in
  dimension $2$, but is still far away from the ultimate conjecture that $R_T
  \asymp T^{3/4}$; see Equation~(1.28) (\textit{ibid.}). To the best of our
  understanding, the only known related finding is presented
  in~\cite{madras:14:lower}, where it is shown that for $d$-dimensional
  self-avoiding random walk with the radius $R_T \gtrsim T^{2/3d}$. It is
  noteworthy that a direct comparison between our outcome and that
  of~\cite{madras:14:lower} is not feasible, as the latter focused on
  self-avoiding random walk, not the self-repellent Brownian motion examined in
  this paper. Additionally, the distance in~\cite{madras:14:lower} pertains
  \textit{end-to-end distance}, which contrasts with that employed in our
  context.
\end{example}

\begin{table}[htpb]
  \centering

  \begin{subtable}{0.99\textwidth}
    \renewcommand{\arraystretch}{1.5}
    \newcommand{\shade}{\cellcolor{gray!30}}
    \begin{tabular}{|c|c|c|c|c|c|} \hline
                                     & $d=1$                     & $d=2$                 & $d=3$                   & $d= 4$                                  & $d\ge 5$       \\ \hline
      Lower bound                    & $1$                       & $1/2$        & $1/3$          & $1/4$                          & $1/d$ \\ \hline
      Conjectured (shaded)/confirmed & $1$                       & \shade $3/4$ & \shade $0.58759700(40)$ & \shade $1/2$ with $\log^{1/8}$ & $1/2$ \\ \hline
      Upper bound                    & \multicolumn{5}{c|}{$1$} \\ \hline
    \end{tabular}
    \caption{The Brownian motion case, i.e., $H=1/2$. The lower and upper bounds
    correspond to the exponent of $T$ in $\underline{R}_T$ and $\overline{R}_T$,
  as defined in~\eqref{E:R-bd}. The conjectured and confirmed values are taken
from~\cite[Table~1]{slade:19:self-avoiding}.}
%\label{Tb:BM}
  \end{subtable}
  \bigskip
  \bigskip

  \begin{subtable}{0.81\textwidth}
    \centering
    \renewcommand{\arraystretch}{1.5}
    \newcommand{\RegEqLe}{\cellcolor{red!20}}
    \newcommand{\RegEqEq}{\cellcolor{red!30}}
    \newcommand{\RegGeLe}{\cellcolor{cyan!35}}
    \newcommand{\RegGeGe}{\cellcolor{cyan!20}}
    \newcommand{\RegLe}{\cellcolor{gray!30}}
    \begin{tabular}{|c|c|c|c|c|c|c|}\hline
      $H$            & $d = 1$                & $d = 2$                                    & $d = 3$                                     & $d = 4$                                  & $d = 5$                                  & $d = 6$                                  \\ \hline
      $\sfrac{1}{4}$ & \RegLe $\sfrac{5}{6}$  & \RegLe   $[\sfrac{7}{16}, \sfrac{13}{16}]$ & \RegLe   $[\sfrac{13}{42}, \sfrac{11}{14}]$ & \RegEqLe $(\sfrac{1}{4}-, \sfrac{3}{4} +)$ & \RegGeLe $[\sfrac{1}{5}, \sfrac{3}{4}]$  & \RegGeLe $[\sfrac{1}{6}, \sfrac{3}{4}]$  \\
      $\sfrac{1}{3}$ & \RegLe $\sfrac{8}{9}$  & \RegLe   $[\sfrac{7}{15}, \sfrac{13}{15}]$ & \RegEqLe $(\sfrac{1}{3}-,   \sfrac{5}{6} + )$ & \RegGeLe $[\sfrac{1}{4}, \sfrac{5}{6}]$  & \RegGeLe $[\sfrac{1}{5}, \sfrac{5}{6}]$  & \RegGeLe $[\sfrac{1}{6}, \sfrac{5}{6}]$  \\
      $\sfrac{1}{2}$ & \RegLe $1$             & \RegEqEq $[\sfrac{1}{2},  1]$              & \RegGeGe $[\sfrac{1}{3},   1]$              & \RegGeGe $[\sfrac{1}{4}, 1]$             & \RegGeGe $[\sfrac{1}{5}, 1]$             & \RegGeGe $[\sfrac{1}{6}, 1]$             \\
      $\sfrac{2}{3}$ & \RegLe $\sfrac{10}{9}$ & \RegGeGe $[\sfrac{5}{9},  \sfrac{10}{9}]$  & \RegGeGe $[\sfrac{10}{27}, \sfrac{10}{9}]$  & \RegGeGe $[\sfrac{5}{18},\sfrac{10}{9}]$ & \RegGeGe $[\sfrac{2}{9}, \sfrac{10}{9}]$ & \RegGeGe $[\sfrac{5}{27},\sfrac{10}{9}]$ \\
      $\sfrac{3}{4}$ & \RegLe $\sfrac{7}{6}$  & \RegGeGe $[\sfrac{7}{12}, \sfrac{7}{6}]$   & \RegGeGe $[\sfrac{7}{18},  \sfrac{7}{6}]$   & \RegGeGe $[\sfrac{7}{24},\sfrac{7}{6}]$  & \RegGeGe $[\sfrac{7}{30},\sfrac{7}{6}]$  & \RegGeGe $[\sfrac{7}{36},\sfrac{7}{6}]$  \\ \hline
    \end{tabular}

    \caption{The ranges (when $d\ge 2$) and the exact values (when $d=1$) for
      various values of $H$ and $d$. A gray background indicates cases with $dH
      < 1$. Red highlights the scenario where $dH = 1$ with darker one for the
      case $d=2$ and lighter one for the cases $d\ge 3$. Cyan represents cases
    where $dH > 1$ with darker one for the cases $H < 1/2$ and lighter one for
  the cases $H\ge 1/2$.}
   \label{Tb:dH}
  \end{subtable}

  \caption{Exponents of $T$ in $R_T$, as define in~\eqref{E:GT}, for large
  $T$.}\label{Tb:exponents}
\end{table}

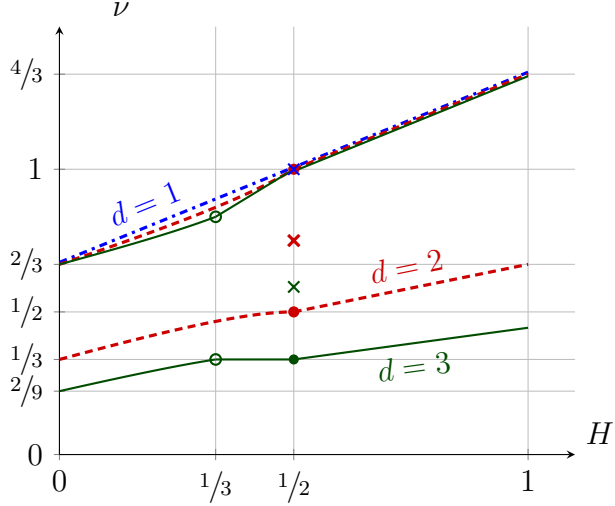
\begin{figure}[htpb]
  \centering
  \begin{tikzpicture}[mark size=1.5pt]

    \begin{axis}[
        xmin=0, xmax=1.1,
        ymin=0, ymax=1.5,
        axis lines=left,
        xlabel={$H$},
        xtick ={0, 1/3, 1/2, 1},
        xticklabels ={$0$, $\sfrac{1}{3}$, $\sfrac{1}{2}$, $1$},
        ylabel ={$\nu$},
        ytick ={0, 2/9,  1/3, 1/2, 2/3, 1, 4/3},
        yticklabels ={$0$, $\sfrac{2}{9}$, $\sfrac{1}{3}$, $\sfrac{1}{2}$, $\sfrac{2}{3}$, $1$, $\sfrac{4}{3}$},
        xlabel style={at={(axis description cs:1,0.05)},anchor=west},
        ylabel style={at={(axis description cs:0.12,1)},anchor=south, rotate = -90},
        grid=major,
    ]
    \def\shift{0.007}

    % d=1
    \addplot[blue, very thick, domain=0:1, samples=100, dashdotted] {2/3*(1 + x)+\shift} node[pos = 0.2, above, sloped] {$d=1$};

    % Lower nu, d=2
    \addplot[red!80!black, very thick, domain=0:0.5, samples=100, densely dashed, no marks] {(1 - 2*x^2)/(3 - 4*x)};
    \addplot[red!80!black, very thick, domain=0.5:1, samples=100, densely dashed, no marks] {(1 + x)/3} node [midway, above, sloped] {$d=2$};
    \addplot[red!80!black, very thick, mark=*, only marks] coordinates {(0.5,0.5)};

    % Upper nu, d=2
    \addplot[red!80!black, very thick, domain=0:0.5, samples=100, densely dashed, no marks] {(2 - x - 2*x^2)/(3 - 4*x)};
    \addplot[red!80!black, very thick, domain=0.5:1, samples=100, densely dashed, no marks] {2/3*(1 + x)};
    \addplot[red!80!black, very thick, mark=*, only marks] coordinates {(0.5,1)};
    \addplot[red!80!black, very thick, mark=x, only marks, mark size = 3pt] coordinates {(0.5,3/4)};

    % Lower nu, d=3
    \addplot[green!30!black, thick, domain=0:1/3, samples=100, no marks] {1/3*(2 - 6*x^2)/(3 - 5*x)};
    \addplot[green!30!black, thick, domain=1/3:1/2, samples=2, no marks] {1/3};
    \addplot[green!30!black, thick, domain=1/2:1, samples=100, no marks] {2/9*(1 + x)} node [midway, below, sloped] {$d=3$};
    \addplot[green!30!black, thick, mark=o, only marks, mark size = 2pt] coordinates {(1/3,1/3)};
    \addplot[green!30!black, thick, mark=*, only marks] coordinates {(1/2,1/3)};

    % Upper nu, d=3
    \addplot[green!30!black, thick, domain=0:1/3, samples=100, no marks] {(2 - 2*x - 2*x^2)/(3 - 5*x)};
    \addplot[green!30!black, thick, domain=1/3:1/2, samples=2, no marks] {(1 + 2*x)/2};
    \addplot[green!30!black, thick, domain=1/2:1, samples=100, no marks] {2/3*(1 + x)-\shift};
    \addplot[green!30!black, thick, mark=o, only marks, mark size = 2pt] coordinates {(1/3,5/6)};
    \addplot[green!30!black, thick, mark=x, only marks, mark size = 3pt] coordinates {(0.5,0.58759700)};

    \addplot[blue, thick, mark=x, only marks, mark size = 3pt] coordinates {(0.5,1)};

    \end{axis}

  \end{tikzpicture}

  \caption{Plots of the exponent $\nu$ for $d=1$ (the blue dash-dotted line),
    $2$ (two red dashed lines, one for the upper bound, the other for the lower
    bound) and $3$ (two green solid lines). The conjectured and confirmed values
    of $\mu$ when $H=1/2$ are labeled by the cross mark (``x'') for the cases
  $d=1$ in blue (confirmed), $d=2$ in red, and $d=3$ in green. The two green
circles for $d=3$ at $H=1/3$ refer to the case that exponent is subject to
logarithm corrections.}

  \label{F:nu-H}
\end{figure}

\begin{remark}
  The Hurst parameter $H$ does not need to be the same for each coordinate. The
  same strategy presented in this paper can be readily applied to other cases.
  Let $H_1,\dots, H_d$ denote the Hurst parameter of $B^{H,1},\dots, B^{H,d}$,
  respectively. Then, the results in Theorem~\ref{T:main} still hold, with
  parameters depending on $H_1 + \dots + H_d$, $\max \{H_1,\dots, H_d\}$, and
  $\min \{H_1,\dots, H_d\}$. For the sake of conciseness, we refrain from
  delving into the specifics in this paper.
\end{remark}

\begin{remark}
  Lemma~\ref{lmm:zt} provides a sharp bound for the Brownian case $H = 1/2$.
  With $B = B^{1/2}$ denoting the $d$-dimensional Brownian motion, we can write
  \begin{align*}
    \int_{\R^d} L_T(y)^2 \ud y
    =    & \int_0^T \ud t_1 \int_0^T \ud t_2 \int_{\R^d} \ud y \: \mathbf{1}_{\mathbf{B}_1(y)} (B_{t_1}) \mathbf{1}_{\mathbf{B}_1(y)} (B_{t_2}) \\
    =    & \int_0^T \ud t_1 \int_0^T \ud t_2 \: \left| \mathbf{B}_1(B_{t_1}) \cap \mathbf{B}_1(B_{t_2}) \right|                                 \\
    \geq & \sum_{k = 0}^{\lfloor T \rfloor} \int_{k}^{k+1} \ud t_1 \int_{k}^{k+1}  \ud t_2 \: \left| \mathbf{B}_1(B_{t_1}) \cap \mathbf{B}_1(B_{t_2}) \right| .
  \end{align*}
  Since $\left| \mathbf{B}_1(B_{t_1}) \cap \mathbf{B}_1(B_{t_2}) \right|$ is a
  non-negative function of $B_{t_1} - B_{t_2}$, the summands in above expression
  are i.i.d. random variables. As a result, with $f (B_{t_1} - B_{t_2})
  \coloneqq \left| \mathbf{B}_1(B_{t_1}) \cap \mathbf{B}_1(B_{t_2}) \right|$,
  \begin{align*}
    Z_T = & \mathbb{E}^{\mathbb{P}} \left[\exp \left( - \beta \int_{\R^d} \ud y \: L_T(y)^2 \right) \right]                                                                            \\
    \leq  & \mathbb{E}^{\mathbb{P}} \left[ \exp \left( - \beta \sum_{k = 0}^{\lfloor T \rfloor} \int_{k}^{k+1} \ud t_1 \int_{k}^{k+1} \ud t_2 \: f (B_{t_1} - B_{t_2}) \right) \right] \\
     =    & \prod_{k = 0}^{\lfloor T \rfloor} \mathbb{E}^{\mathbb{P}} \left[\exp \left( - \beta  \int_{k}^{k + 1} \ud t_1 \int_k^{k + 1} \ud t_2 \:  f (B_{t_1} - B_{t_2}) \right) \right] = e^{F_{\beta}\lfloor T \rfloor},
  \end{align*}
  where
  \begin{align*}
    F_{\beta} \coloneqq \log \mathbb{E}^{\mathbb{P}} \left[\exp \left( - \beta  \int_0^1 \ud t_1 \int_0^1 \ud t_2 \:  f (B_{t_1} - B_{t_2}) \right) \right] < 0.
  \end{align*}
  In other words, $Z_T \leq \exp(- C T)$ with some $C>0$ for all $T > 1$, and
  combining Lemma~\ref{lmm:zt}, we see that $\log Z_T \asymp - T$ as $T \to
  \infty$. Therefore, the lack of sharpness in Theorem~\ref{T:main} is likely
  attributed to the estimates in Lemma~\ref{lmm:q}. We expect that this aspect
  can be resolved in future research.
\end{remark}

\appendix

\section{Fractional Brownian motions and Girsanov theorem}

In this section, we present some preliminaries about stochastic calculus for
fBm's. For a more detailed account of this topic, we refer the interested
readers to~\cite{biagini.hu.ea:08:stochastic}. \medskip

A $d$-dimensional stochastic process $B^H = \left\{\left(B^{H,1}_t,\dots,
B^{H,d}_t\right) \colon t\in \mathbb{R}_+\right\}$ is a called a
\textit{fraction Brownian motion} (\textit{fBm}) with the Hurst parameters $H
\in (0,1)$ on a probability space $(\Omega, \mathcal{F}, \mathbb{P})$, if
\begin{enumerate}[(i)]
  \item $B^{H,i}$, $i=1,\dots,d$, are independent;
  \item for each $1\leq i\leq d$, $\left\{B^{H,i}_t \colon t \in
    \mathbb{R}_+\right\}$ is a centered Gaussian family with covariance
  \begin{align*}
    \mathbb{E} \left[ B^{H,i}_t B^{H,i}_s\right]
    = \frac{1}{2}\left( t^{2H} + s^{2H} - |t-s|^{2H}\right).
  \end{align*}
\end{enumerate}
Without loss of generality, we can assume that the filtration
$\left\{\mathcal{F}_t \colon t \in \mathbb{R}_+\right\}$ is the canonical
filtration generated by $B^H$.

Next, we define the integration of deterministic functions against fBm's. If
$\phi$ is a smooth function on $\mathbb{R}_+$ with compact support, i.e., $\phi
\in C_c^{\infty}(\mathbb{R}_+)$, then the integrals
\begin{align*}
   B^{H,i}(\phi) \coloneqq \int_0^{\infty} \phi(t) \ud B^{H,i}_t, \qquad i = 1,\dots, d,
\end{align*}
are centered Gaussian random variables with the following covariance structure:
\begin{align*}
  \mathbb{E} \left[B^{H,i}(\phi) B^{H,j}(\psi) \right]
  = \begin{dcases}
  0,                                                                   & i \neq j, \\
  H(2H-1)\iint_{\R_+^2} \ud t \: \ud s\: \phi(s)\psi(t) |t-s|^{2H - 2}, & i = j,
  \end{dcases}
\end{align*}
for all $1\leq i,j \leq d$, and $\phi, \psi \in C_c^{\infty}(\mathbb{R}_+).$

By typical approximation arguments, one can extend the integration to the
Hilbert space $\mathcal{H}$ of functions on $\mathbb{R}_+$, with inner product,
\begin{align*}
  \InPrd{\phi, \psi}_{\mathcal{H}}
  \coloneqq \iint_{\R_+^2} \ud t \: \ud s\: \phi(s)\psi(t) |t-s|^{2H - 2}.
\end{align*}
In particular, for any $t\in \mathbb{R}_+$, the function $w(t,\cdot )$, given by
\begin{gather*}
  w (t,s) \coloneqq c_{1}\, s^{1/2 - H} (t - s)^{1/2 - H} \mathbf{1}_{(0,t)} (s),\quad \text{for all } s\in \mathbb{R},
\end{gather*}
is an element of $\mathcal{H}$ (see~\cite[Proposition
2.1]{norros.valkeila.ea:99:elementary}), where, with $B(\cdot,\cdot)$ denoting
the \textit{Beta function},
\begin{gather*}
  c_{1} \coloneqq \left[ 2 H \times B \left(\sfrac{3}{2} - H, \sfrac{1}{2} + H\right)\right ]^{-1}.
\end{gather*}
This observation allows us to define the following Gaussian process $M =
M^{\mathbf{u}} = \{ M_t \colon t \in \R_+\}$ with parameter $\mathbf{u} =
(u_1,\dots, u_d)$ being a unit vector in $\R^d$ as follows:
\begin{align}\label{E:M}
  M_t \coloneqq \sum_{i=1}^d u_i \int_0^t w(t,s) \ud B^{H,i}_s, \quad \text{for all } t\in \mathbb{R}_+.
\end{align}
The following theorem is a straightforward extension
of~\cite[Theorem~3.1]{norros.valkeila.ea:99:elementary} from $d=1$ to higher
dimensional cases, thanks to the independence of the components of $B^H$:
\medskip

\begin{theorem}\label{T:mar}
  Let $B^H$ be a $d$-dimensional fBm with $H \in (0,1)$, let $\mathbf{u}$ be a
  unit vector in $\R^d$, and let $M = M^{\mathbf{u}}$ be the Gaussian process
  given as in~\eqref{E:M}. Then $M$ is a square-integrable martingale with
  quadratic variation
  \begin{align*}%\label{E:mar}
    \InPrd{M, M}_t
    = C_H t^{2(1 - H)},\ \forall t\in \R_+, \quad \text{where}\quad
      C_H \coloneqq \frac{\Gamma (\sfrac{3}{2} - H)}{4 H (1 - H)\Gamma(\sfrac{1}{2} + H) \Gamma(2 - 2H)}.
  \end{align*}
\end{theorem}

For any $\lambda > 0$ and $T > 0$, denote
\begin{align*}
  \mathcal{Q}_T (M) \coloneqq \exp \left( M_T - \frac{1}{2} \langle M, M \rangle_T \right),
\end{align*}
and let $\mathbb{P}_T^{\lambda} = \mathbb{P}_T^{\lambda, \mathbf{u}}$ be a
probability measure on $(\Omega, \mathcal{F}_T)$ that is equivalent to
$\mathbb{P}_T$ with the Radon--Nikodym derivative
\begin{align}\label{E:p-lambda}
  \frac{\ud \mathbb{P}_T^{\lambda}}{\ud \mathbb{P}_T} \coloneqq \mathcal{Q}_T(\lambda M).
\end{align}

The next theorem, a \textit{Girsanov formula for fBm's}, is a straightforward
extension of~\cite[Theorem~4.1]{norros.valkeila.ea:99:elementary}.

\begin{theorem}\label{T:gir}
  Under probability $\mathbb{P}^{\lambda}_T$, the process $\left\{B^H_t \colon 0
  \leq t \leq T\right\}$ is a $d$-dimensional fBm with a drift
  $\lambda\mathbf{u}\in\R^d$, i.e., the distribution of the process $B^H$ up to
  time $T$ under $\mathbb{P}^{\lambda}_T = \mathbb{P}^{\lambda,\mathbf{u}}_T$ is
  the same as $ B^{H, \lambda, \mathbf{u}, T} = \left\{B^H_t + \lambda t\,
  \mathbf{u}  \:\colon 0\leq t\leq T\right\}$ under $\mathbb{P}_T$.
\end{theorem}

\section*{Acknowledgments}

L.~C. and P.~X. are partially supported by NSF grant DMS-2246850 and L.~C. is
partially supported by a collaboration grant (\#959981) from the Simons
foundation.


\begin{thebibliography}{22}
\ifx \bisbn   \undefined \def \bisbn  #1{ISBN #1}\fi
\ifx \binits  \undefined \def \binits#1{#1}\fi
\ifx \bauthor  \undefined \def \bauthor#1{#1}\fi
\ifx \batitle  \undefined \def \batitle#1{#1}\fi
\ifx \bjtitle  \undefined \def \bjtitle#1{#1}\fi
\ifx \bvolume  \undefined \def \bvolume#1{\textbf{#1}}\fi
\ifx \byear  \undefined \def \byear#1{#1}\fi
\ifx \bissue  \undefined \def \bissue#1{#1}\fi
\ifx \bfpage  \undefined \def \bfpage#1{#1}\fi
\ifx \blpage  \undefined \def \blpage #1{#1}\fi
\ifx \burl  \undefined \def \burl#1{\textsf{#1}}\fi
\ifx \doiurl  \undefined \def \doiurl#1{\url{https://doi.org/#1}}\fi
\ifx \betal  \undefined \def \betal{\textit{et al.}}\fi
\ifx \binstitute  \undefined \def \binstitute#1{#1}\fi
\ifx \binstitutionaled  \undefined \def \binstitutionaled#1{#1}\fi
\ifx \bctitle  \undefined \def \bctitle#1{#1}\fi
\ifx \beditor  \undefined \def \beditor#1{#1}\fi
\ifx \bpublisher  \undefined \def \bpublisher#1{#1}\fi
\ifx \bbtitle  \undefined \def \bbtitle#1{#1}\fi
\ifx \bedition  \undefined \def \bedition#1{#1}\fi
\ifx \bseriesno  \undefined \def \bseriesno#1{#1}\fi
\ifx \blocation  \undefined \def \blocation#1{#1}\fi
\ifx \bsertitle  \undefined \def \bsertitle#1{#1}\fi
\ifx \bsnm \undefined \def \bsnm#1{#1}\fi
\ifx \bsuffix \undefined \def \bsuffix#1{#1}\fi
\ifx \bparticle \undefined \def \bparticle#1{#1}\fi
\ifx \barticle \undefined \def \barticle#1{#1}\fi
% \bibcommenthead
\ifx \bconfdate \undefined \def \bconfdate #1{#1}\fi
\ifx \botherref \undefined \def \botherref #1{#1}\fi
\ifx \url \undefined \def \url#1{\textsf{#1}}\fi
\ifx \bchapter \undefined \def \bchapter#1{#1}\fi
\ifx \bbook \undefined \def \bbook#1{#1}\fi
\ifx \bcomment \undefined \def \bcomment#1{#1}\fi
\ifx \oauthor \undefined \def \oauthor#1{#1}\fi
\ifx \citeauthoryear \undefined \def \citeauthoryear#1{#1}\fi
\ifx \endbibitem  \undefined \def \endbibitem {}\fi
\ifx \bconflocation  \undefined \def \bconflocation#1{#1}\fi
\ifx \arxivurl  \undefined \def \arxivurl#1{\textsf{#1}}\fi
\csname PreBibitemsHook\endcsname

%%% 1
\bibitem{domb.joyce:72:cluster}
\begin{barticle}
\bauthor{\bsnm{Domb}, \binits{C.}},
\bauthor{\bsnm{Joyce}, \binits{G.S.}}:
\batitle{Cluster expansion for a polymer chain}.
\bjtitle{J. Phys. C: Solid State Phys.}
\bvolume{5}(\bissue{9}),
\bfpage{956}
(\byear{1972})
\doiurl{10.1088/0022-3719/5/9/009}
\end{barticle}
\endbibitem

%%% 2
\bibitem{edwards:65:statistical}
\begin{barticle}
\bauthor{\bsnm{Edwards}, \binits{S.F.}}:
\batitle{The statistical mechanics of polymers with excluded volume}.
\bjtitle{Proc. Phys. Soc.}
\bvolume{85},
\bfpage{613}--\blpage{624}
(\byear{1965})
\doiurl{10.1007/s001090000086}
\end{barticle}
\endbibitem

%%% 3
\bibitem{madras.slade:93:self-avoiding}
\begin{bbook}
\bauthor{\bsnm{Madras}, \binits{N.}},
\bauthor{\bsnm{Slade}, \binits{G.}}:
\bbtitle{The Self-avoiding Walk}.
\bsertitle{Probability and its Applications},
p. \bfpage{425}.
\bpublisher{Birkh{\"{a}}user Boston, Inc.},
\blocation{Boston}
(\byear{1993})
\end{bbook}
\endbibitem

%%% 4
\bibitem{bauerschmidt.duminil-copin.ea:12:lectures}
\begin{bchapter}
\bauthor{\bsnm{Duminil-Copin}, \binits{H.}},
\bauthor{\bsnm{Goodman}, \binits{J.}},
\bauthor{\bsnm{Slade}, \binits{G.}}:
\bctitle{Lectures on self-avoiding walks}.
In: \bbtitle{Probability and Statistical Physics in Two and More Dimensions}.
\bsertitle{Clay Math. Proc.},
vol. \bseriesno{15},
pp. \bfpage{395}--\blpage{467}.
\bpublisher{Amer. Math. Soc.},
\blocation{Providence}
(\byear{2012})
\end{bchapter}
\endbibitem

%%% 5
\bibitem{slade:19:self-avoiding}
\begin{barticle}
\bauthor{\bsnm{Slade}, \binits{G.}}:
\batitle{Self-avoiding walk, spin systems and renormalization}.
\bjtitle{Proc. A.}
\bvolume{475}(\bissue{2221}),
\bfpage{20180549}--\blpage{21}
(\byear{2019})
\doiurl{10.1098/rspa.2018.0549}
\end{barticle}
\endbibitem

%%% 6
\bibitem{biswas.cherayil:95:dynamics}
\begin{barticle}
\bauthor{\bsnm{Biswas}, \binits{P.}},
\bauthor{\bsnm{Cherayil}, \binits{B.J.}}:
\batitle{Dynamics of fractional brownian walks}.
\bjtitle{J. Phys. Chem.}
\bvolume{99}(\bissue{2}),
\bfpage{816}--\blpage{821}
(\byear{1995})
\doiurl{10.1021/j100002a052}
\end{barticle}
\endbibitem

%%% 7
\bibitem{grothaus.oliveira.ea:11:self-avoiding}
\begin{barticle}
\bauthor{\bsnm{Grothaus}, \binits{M.}},
\bauthor{\bsnm{Oliveira}, \binits{M.J.}},
\bauthor{\bsnm{Silva}, \binits{J.L.}},
\bauthor{\bsnm{Streit}, \binits{L.}}:
\batitle{Self-avoiding fractional {B}rownian motion---the {E}dwards model}.
\bjtitle{J. Stat. Phys.}
\bvolume{145}(\bissue{6}),
\bfpage{1513}--\blpage{1523}
(\byear{2011})
\doiurl{10.1007/s10955-011-0344-2}
\end{barticle}
\endbibitem

%%% 8
\bibitem{bornales.oliveira.ea:13:self-repelling}
\begin{bchapter}
\bauthor{\bsnm{Bornales}, \binits{J.}},
\bauthor{\bsnm{Oliveira}, \binits{M.J.}},
\bauthor{\bsnm{Streit}, \binits{L.}}:
\bctitle{Self-repelling fractional {B}rownian motion---a generalized {E}dwards
  model for chain polymers}.
In: \bbtitle{Quantum Bio-informatics {V}}.
\bsertitle{QP--PQ: Quantum Probab. White Noise Anal.},
vol. \bseriesno{30},
pp. \bfpage{389}--\blpage{401}.
\bpublisher{World Sci. Publ.},
\blocation{Hackensack}
(\byear{2013})
\end{bchapter}
\endbibitem

%%% 9
\bibitem{bock.bornales.ea:15:scaling}
\begin{barticle}
\bauthor{\bsnm{Bock}, \binits{W.}},
\bauthor{\bsnm{Bornales}, \binits{J.B.}},
\bauthor{\bsnm{Cabahug}, \binits{C.O.}},
\bauthor{\bsnm{Eleut\'{e}rio}, \binits{S.}},
\bauthor{\bsnm{Streit}, \binits{L.}}:
\batitle{Scaling properties of weakly self-avoiding fractional {B}rownian
  motion in one dimension}.
\bjtitle{J. Stat. Phys.}
\bvolume{161}(\bissue{5}),
\bfpage{1155}--\blpage{1162}
(\byear{2015})
\doiurl{10.1007/s10955-015-1368-9}
\end{barticle}
\endbibitem

%%% 10
\bibitem{rosen:87:intersection}
\begin{barticle}
\bauthor{\bsnm{Rosen}, \binits{J.}}:
\batitle{The intersection local time of fractional {B}rownian motion in the
  plane}.
\bjtitle{J. Multivariate Anal.}
\bvolume{23}(\bissue{1}),
\bfpage{37}--\blpage{46}
(\byear{1987})
\doiurl{10.1016/0047-259X(87)90176-X}
\end{barticle}
\endbibitem

%%% 11
\bibitem{hu.nualart:05:renormalized}
\begin{barticle}
\bauthor{\bsnm{Hu}, \binits{Y.}},
\bauthor{\bsnm{Nualart}, \binits{D.}}:
\batitle{Renormalized self-intersection local time for fractional {B}rownian
  motion}.
\bjtitle{Ann. Probab.}
\bvolume{33}(\bissue{3}),
\bfpage{948}--\blpage{983}
(\byear{2005})
\doiurl{10.1214/009117905000000017}
\end{barticle}
\endbibitem

%%% 12
\bibitem{mueller.neuman:23:radius}
\begin{botherref}
\oauthor{\bsnm{Mueller}, \binits{C.}},
\oauthor{\bsnm{E.}, \binits{N.}}:
The radius of a self-repelling star polymer.
preprint arXiv:2306.01537
(2023)
\end{botherref}
\endbibitem

%%% 13
\bibitem{fixman:62:radius}
\begin{barticle}
\bauthor{\bsnm{Fixman}, \binits{M.}}:
\batitle{Radius of gyration of polymer chains}.
\bjtitle{J. Chem. Phys.}
\bvolume{36}(\bissue{2}),
\bfpage{306}--\blpage{310}
(\byear{1962})
\doiurl{10.1063/1.1732501}
\end{barticle}
\endbibitem

%%% 14
\bibitem{bolthausen:90:on}
\begin{barticle}
\bauthor{\bsnm{Bolthausen}, \binits{E.}}:
\batitle{On self-repellent one-dimensional random walks}.
\bjtitle{Probab. Theory Related Fields}
\bvolume{86}(\bissue{4}),
\bfpage{423}--\blpage{441}
(\byear{1990})
\doiurl{10.1007/BF01198167}
\end{barticle}
\endbibitem

%%% 15
\bibitem{greven.hollander:93:variational}
\begin{barticle}
\bauthor{\bsnm{Greven}, \binits{A.}},
\bauthor{\bsnm{Hollander}, \binits{F.}}:
\batitle{A variational characterization of the speed of a one-dimensional
  self-repellent random walk}.
\bjtitle{Ann. Appl. Probab.}
\bvolume{3}(\bissue{4}),
\bfpage{1067}--\blpage{1099}
(\byear{1993})
\doiurl{10.1214/aoap/1177005273}
\end{barticle}
\endbibitem

%%% 16
\bibitem{hara.slade:92:self-avoiding}
\begin{barticle}
\bauthor{\bsnm{Hara}, \binits{T.}},
\bauthor{\bsnm{Slade}, \binits{G.}}:
\batitle{Self-avoiding walk in five or more dimensions. {I}. {T}he critical
  behaviour}.
\bjtitle{Comm. Math. Phys.}
\bvolume{147}(\bissue{1}),
\bfpage{101}--\blpage{136}
(\byear{1992})
\doiurl{10.1007/BF02099530}
\end{barticle}
\endbibitem

%%% 17
\bibitem{brydges.spencer:85:self-avoiding}
\begin{barticle}
\bauthor{\bsnm{Brydges}, \binits{D.}},
\bauthor{\bsnm{Spencer}, \binits{T.}}:
\batitle{Self-avoiding walk in {$5$} or more dimensions}.
\bjtitle{Comm. Math. Phys.}
\bvolume{97}(\bissue{1-2}),
\bfpage{125}--\blpage{148}
(\byear{1985})
\doiurl{10.1007/BF01206182}
\end{barticle}
\endbibitem

%%% 18
\bibitem{adler.taylor:07:random}
\begin{bbook}
\bauthor{\bsnm{Adler}, \binits{R.J.}},
\bauthor{\bsnm{Taylor}, \binits{J.E.}}:
\bbtitle{Random Fields and Geometry}.
\bsertitle{Springer Monographs in Mathematics},
p. \bfpage{448}.
\bpublisher{Springer},
\blocation{New York}
(\byear{2007})
\end{bbook}
\endbibitem

%%% 19
\bibitem{hara.slade:91:critical}
\begin{barticle}
\bauthor{\bsnm{Hara}, \binits{T.}},
\bauthor{\bsnm{Slade}, \binits{G.}}:
\batitle{Critical behaviour of self-avoiding walk in five or more dimensions}.
\bjtitle{Bull. Amer. Math. Soc. (N.S.)}
\bvolume{25}(\bissue{2}),
\bfpage{417}--\blpage{423}
(\byear{1991})
\doiurl{10.1090/S0273-0979-1991-16085-4}
\end{barticle}
\endbibitem

%%% 20
\bibitem{madras:14:lower}
\begin{barticle}
\bauthor{\bsnm{Madras}, \binits{N.}}:
\batitle{A lower bound for the end-to-end distance of the self-avoiding walk}.
\bjtitle{Canad. Math. Bull.}
\bvolume{57}(\bissue{1}),
\bfpage{113}--\blpage{118}
(\byear{2014})
\doiurl{10.4153/CMB-2012-022-6}
\end{barticle}
\endbibitem

%%% 21
\bibitem{biagini.hu.ea:08:stochastic}
\begin{bbook}
\bauthor{\bsnm{Biagini}, \binits{F.}},
\bauthor{\bsnm{Hu}, \binits{Y.}},
\bauthor{\bsnm{{\O}ksendal}, \binits{B.}},
\bauthor{\bsnm{Zhang}, \binits{T.}}:
\bbtitle{Stochastic Calculus for Fractional {B}rownian Motion and
  Applications}.
\bsertitle{Probability and its Applications (New York)},
p. \bfpage{329}.
\bpublisher{Springer},
\blocation{London}
(\byear{2008})
\end{bbook}
\endbibitem

%%% 22
\bibitem{norros.valkeila.ea:99:elementary}
\begin{barticle}
\bauthor{\bsnm{Norros}, \binits{I.}},
\bauthor{\bsnm{Valkeila}, \binits{E.}},
\bauthor{\bsnm{Virtamo}, \binits{J.}}:
\batitle{An elementary approach to a {G}irsanov formula and other analytical
  results on fractional {B}rownian motions}.
\bjtitle{Bernoulli}
\bvolume{5}(\bissue{4}),
\bfpage{571}--\blpage{587}
(\byear{1999})
\doiurl{10.2307/3318691}
\end{barticle}
\endbibitem

\end{thebibliography}
\end{document}